\documentclass[11pt,a4paper]{article}
\usepackage{fullpage}
\usepackage[utf8]{inputenc}
\usepackage{amsthm,amssymb}
\usepackage[leqno]{amsmath}
\usepackage{tcolorbox}
\usepackage{mathtools}
\usepackage{thm-restate}
\usepackage{enumerate}
\usepackage[sort]{cite}
\usepackage{titling,authblk}
\usepackage{xparse}
\usepackage{hyperref}
\hypersetup{
colorlinks=true,       
linkcolor=blue,        
citecolor=red,         
filecolor=magenta,     
urlcolor=blue,         
linktocpage=true
}

\usepackage{float}
\usepackage[linesnumbered,ruled,vlined]{algorithm2e}
\usepackage{enumitem}

\newenvironment{symenum}
{\enumerate[label=\textbf{\noexpand\thisenumsymbol}, align=left, labelindent=0pt, itemindent=-40pt, labelsep=5pt, leftmargin=*]}
{\endenumerate}
\newcommand\thisenumsymbol{}
\newcommand\itemsymbol[1]{\renewcommand{\thisenumsymbol}{#1}\item
}

\allowdisplaybreaks[4] 

\theoremstyle{plain}
\newtheorem{thm}{Theorem}
\newtheorem{lemma}[thm]{Lemma}
\newtheorem{cor}[thm]{Corollary}

\theoremstyle{definition}

\newcounter{casenum}

\DeclareMathOperator{\ei}{\ell_{\infty}}

\makeatletter
\NewDocumentCommand{\makemathbox}{O{\width} m}{%
  \def\makemathbox@##1##2{\makebox[#1]{$##1##2$}}%
  \mathpalette\makemathbox@{#2}%
}
\makeatother

\newcommand{\p}[3]{\ifthenelse{\equal{#3}{}}{p_{[#1|#2]}}{p_{[#1|#2|#3]}}}
\newcommand{\cO}{\mathcal{O}}
\newcommand{\cF}{\mathcal{F}}

\def\final{0}  
\def\iflong{\iffalse}
\ifnum\final=0  
\newcommand{\krnote}[1]{{\color{red}[{\tiny \textbf{Krist{\'o}f:} \bf #1}]\marginpar{\color{red}*}}}
\newcommand{\kinote}[1]{{\color{teal}[{\tiny \textbf{Kitti:} \bf #1}]\marginpar{\color{teal}*}}}
\newcommand{\lnote}[1]{{\color{purple}[{\tiny \textbf{Lydia:} \bf #1}]\marginpar{\color{purple}*}}}
\else 
\newcommand{\krnote}[1]{}
\newcommand{\kinote}[1]{}
\newcommand{\lnote}[1]{}
\fi

\title{Newton-type algorithms for inverse optimization I:\\
weighted bottleneck Hamming distance and $\ell_\infty$-norm objectives}
\date{}

\thanksmarkseries{alph}
\author{Krist{\'o}f B{\'e}rczi} 
\author{Lydia Mirabel Mendoza-Cadena}
\author{Kitti Varga}
\affil{{\footnotesize MTA-ELTE Momentum Matroid Optimization Research Group and ELKH-ELTE Egerv\'ary Research Group, Department of Operations Research, E\"otv\"os Lor\'and University, Budapest, Hungary. Email: \texttt{kristof.berczi@ttk.elte.hu, lmmendoza@proton.me, vkitti@math.bme.hu}.}}

\begin{document}
\maketitle

\begin{abstract}
In minimum-cost inverse optimization problems, we are given a feasible solution to an underlying optimization problem together with a linear cost function, and the goal is to modify the costs by a small deviation vector so that the input solution becomes optimal.

The difference between the new and the original cost functions can be measured in several ways. In this paper, we focus on two objectives: the weighted bottleneck Hamming distance and the weighted $\ei$-norm. We consider a general model in which the coordinates of the deviation vector are required to fall within given lower and upper bounds. For the weighted bottleneck Hamming distance objective, we present a simple, purely combinatorial algorithm that determines an optimal deviation vector in strongly polynomial time. For the weighted $\ei$-norm objective, we give a min-max characterization for the optimal solution, and provide a pseudo-polynomial algorithm for finding an optimal deviation vector that runs in strongly polynomial time in the case of unit weights. For both objectives, we assume that an algorithm with the same time complexity for solving the underlying combinatorial optimization problem is available. 

For both objectives, we also show how to extend the results to inverse optimization problems with multiple cost functions.
\medskip


\end{abstract}

\section{Introduction}
\label{sec:intro}

Inverse optimization problems have long been the focus of research due to their wide applicability in both theory and practice. The roots of inverse optimization go back to the work of Burton and Toit~\cite{burton1992instance} who studied the inverse shortest paths problem, that is, the problem of recovering the edge costs given some information about the shortest paths in the graph. Since their pioneering work, countless of applications and extensions emerged; we refer the interested reader to \cite{richter2016inverse} for the basics and to \cite{heuberger2004inverse,demange2014introduction} for surveys. 

In a classical optimization problem, we are given a set of feasible solutions together with a linear cost function, and the goal is to find a feasible solution that minimizes or maximizes the cost. In contrast, in an inverse optimization problem we are also given a fixed feasible solution, and the goal is to modify the costs `as little as possible' so that the input solution becomes optimal. There are various ways to measure the deviation of the new cost function from the original one, and, as one would expect, the choice of the objective greatly affects the complexity of the problem. In order to avoid confusion, we refer to the solutions of the inverse optimization problem and those of the underlying combinatorial optimization problem as \emph{feasible deviation vectors} and \emph{solutions}, respectively.

In the past decades, inverse optimization problems found numerous applications. As an example, let us briefly describe the \emph{Pathway concordance problem}~\cite{chan2021Pathway}. A clinical pathway describes a standardized sequence of steps for managing a clinical process in the delivery of care for a specific disease, with the aim of optimizing the outcome on a patient or population-level. These processes are determined by multidisciplinary medical experts, and have been shown to efficiently improve e.g.\ patient survival and satisfaction, wait times, and cost of care. However, patients' journeys through the healthcare system can differ significantly from the recommended pathways, which raises the problem of measuring the concordance of patient-traversed pathways against the recommended ones. The problem can be modeled by a directed graph whose vertices correspond to activities that the patient can undertake, and the arcs indicate that a patient went from one activity to another. The `cost' of a patient undertaking or missing certain activities and traversing arcs can be modeled by arc costs. The goal is to determine arc costs such that the reference pathways are optimal, that is, they are shortest paths between the corresponding start and end vertices. Then, assuming such arc costs are available, the journey of any patient can be scored based on the cost of the associated directed walk through the network. 

The present work is the first member of a series of papers. Our general goal is to give min-max characterizations and simple algorithms for inverse optimization problems under various objectives. Here we set up the basics of a general framework that uses a Newton-type approach for finding an optimal deviation vector, and derive algorithms for the weighted bottleneck Hamming distance and weighted $\ei$-norm objectives that follow the proposed scheme. 

In the second part~\cite{berczi2023span}, we focus on a novel objective called the weighted span that aims at finding a `balanced' or `fair' deviation vector, and we propose an analogous algorithm for that objective. Nevertheless, the analysis of the algorithm there is much more involved due to the different nature of the span compared to the $\ei$-norm. 

\paragraph*{Previous work.}

Inverse optimization under the weighted bottleneck Hamming distance objective\footnote{This objective is sometimes called weighted \emph{bottleneck-type} Hamming distance objective in the literature.} has been of great interest recently. One of the earliest results is due to Duin and Volgenant~\cite{duin2006inverse} who considered the inverse minimum spanning tree, the inverse shortest path tree and the linear assignment problems. Liu and Yao~\cite{liu2013weighted} gave a strongly polynomial algorithm for the weighted inverse maximum perfect matching problem. An algorithm with an improved running time based on a binary search technique was later given by Tayyebi~\cite{tayyebi2019inverse}, and an analogous algorithm for the inverse matroid problem was given by Aman, Hassanpour and Tayyebi~\cite{aman2016matroid}. Mohaghegh and Baroughi Bonab~\cite{mohaghegh2016inverse} showed that the inverse min-max spanning $r$-arborescence problem is solvable in strongly polynomial time. Guan, He, Pardalos and Zhang~\cite{guan2017inverse} presented a mathematical model for the inverse max+sum spanning tree problem, together with a method to check feasibility and a binary search algorithm for solving it. Karimi, Aman and Dolati~\cite{karimi2017multi} studied the inverse shortest $s$-$t$ path problem and provided an LP-based algorithm which can be applied for some inverse multiobjective problem as well. Tayyebi and Aman~\cite{tayyebi2018bottleneck} considered a general inverse linear programming problem, and proposed an algorithm that is based on a binary search technique. As an application, they specialized the method for solving the corresponding inverse minimum-cost flow problem in strongly polynomial time. Nguyen and Hung~\cite{nguyen2020pmedian} studied the so-called inverse connected $p$-median problem under the unweighted bottleneck Hamming distance objective. In this problem, the goal is to modify vertex weights of a block graph at minimum total cost so that a predetermined set of $p$ connected vertices becomes a connected $p$-median on the perturbed block graph. They formulated the problem as a quasiconvex univariate optimization problem, and developed a combinatorial algorithm that solves the problem in polynomial time. Jiang, Liu and Peng~\cite{jiang2021constrained} presented a strongly polynomial algorithm for the inverse minimum flow problem. Dong, Li and Yang~\cite{dong2022partial} addressed the partial inverse min-max spanning tree problem, and presented two algorithms to solve the problem in polynomial time. 

Inverse problems under the $\ell_{\infty}$-norm have been studied in various settings. Xiaoguang~\cite{xiaoguang1998note} considered the inverse optimization problem of submodular functions on digraphs, and gave an LP-based algorithm that solves most inverse network optimization problems in polynomial time. Zhang and Liu~\cite{ZHANG1999Further} suggested a method for solving a general inverse LP problem including upper and lower bound constraints. In a later paper~\cite{liu2003inverse}, the same authors studied the inverse maximum-weight matching problem in non-bipartite graphs under the $\ell_\infty$-norm objective. They showed that the problem can be formulated as a maximum-mean alternating cycle problem in an undirected network, and can be solved in polynomial time by a binary search algorithm and in strongly polynomial time by an ascending algorithm. Using LP descriptions, Ahuja and Orlin~\cite{Ahuja2001Inverse} proved that if an optimization problem can be modeled as an LP, then the same holds for the underlying inverse optimization problem under $\ell_1$- or $\ell_\infty$-norm objectives. Furthermore, if the optimization problem is polynomially solvable for linear cost functions, then the inverse counterparts with $\ell_1$- and $\ell_{\infty}$-norms are also polynomially solvable. In \cite{zhang2002general}, Zhang and Liu proposed a model that generalizes numerous inverse combinatorial optimization problems when no bounds are given on the coordinates of the deviation vector. Yang and Zhang~\cite{yang2007some} presented strongly polynomial algorithms to solve the inverse min-max spanning tree and the inverse maximum capacity path problems when bounds are also given on the coordinates of the deviation vector. Lasserre~\cite{lasserre2013inverse} considered the inverse optimization problem associated with the polynomial program and a given current feasible solution, and provided a systematic numerical scheme to compute an inverse optimal solution. Ahmadian, Bhaskar, Sanit\`a, and Swamy~\cite{ahmadian2018algorithms} studied integral inverse optimization problems from an approximation point of view. They obtained tight or nearly-tight approximation guarantees for various inverse optimization problems, and some of their results apply for $\ell_{\infty}$-norm as well. Zhang, Guan, and Zhang~\cite{zhang2020inverse} provided a mathematical model of the inverse spanning tree problem, gave a characterization of optimal solutions, and developed a strongly polynomial algorithm for determining an optimal deviation vector. Recently, the authors~\cite{berczi2022multiple} introduced inverse optimization problems with multiple cost functions, and studied the inverse minimum-cost $s$-$t$ path, $r$-arborescence, and bipartite perfect matching problems.

Most papers on inverse optimization consider algorithmic aspects, and so they do not provide a min-max characterization for the optimum value in question. Recently, Frank and Murota~\cite{frank2021discrete} developed a general min-max formula for the minimum of an integer-valued separable discrete convex function, where the minimum is taken over the set of integral elements of a box total dual integral polyhedron. Their approach covers and even extends a wide class of inverse combinatorial optimization problems. Nevertheless, our problems do not fit in the box-TDI framework as neither the bottleneck Hamming distance nor the $\ell_\infty$-norm is separable convex.

\paragraph*{Problem definitions.}

We denote the sets of \emph{real} and \emph{positive real} numbers by $\mathbb{R}$ and $\mathbb{R}_+$, respectively. For a positive integer $k$, we use $[k]\coloneqq \{1,\dots,k\}$. Let $S$ be a ground set of size~$n$. Given subsets $X,Y\subseteq S$, the \emph{symmetric difference} of $X$ and $Y$ is denoted by $X\triangle Y\coloneqq (X\setminus Y)\cup(Y\setminus X)$. For a weight function $w\in\mathbb{R}_+^S$, the total sum of its values over $X$ is denoted by $w(X)\coloneqq \sum \{ w(s) \mid s\in X \}$, where the sum over the empty set is always considered to be $0$. Furthermore, we define $\frac{1}{w}(X)\coloneqq \sum \big\{ \frac{1}{w(s)} \bigm| s\in X \big\}$, and set $\|w\|_{-1}\coloneqq \frac{1}{w}(S)$. When the weights are rational numbers, then the values can be re-scaled as to satisfy $1/w(s)$ being an integer for each $s\in S$. Throughout the paper, we assume that $w$ is given in such a form  without explicitly mentioning it, implying that $\frac{1}{w}(X)$ is a non-negative integer for every $X\subseteq S$. By convention, we define $\min\{\emptyset\}=+\infty$ and $\max\{\emptyset\}=-\infty$.

Let $S$ be a finite ground set, $\cF \subseteq 2^S$ be a collection of \emph{feasible solutions} for an underlying optimization problem, $F^* \in \cF$ be an \emph{input solution}, $c \in \mathbb{R}^S$ be a \emph{cost function}, $w\in \mathbb{R}_+^S$ be a positive \emph{weight function}, and $\ell\colon S\to\mathbb{R}\cup\{-\infty\}$ and $u \colon S\to\mathbb{R}\cup\{+\infty\}$ be \emph{lower and upper bounds}, respectively, such that $\ell \leq u $. We assume that an \emph{oracle} $\cO$ is also available that determines an optimal solution of the underlying optimization problem $(S, \cF, c')$ for any cost function $c'\in\mathbb{R}^S$.

In the \emph{constrained minimum-cost inverse optimization problem under the weighted bottleneck Hamming distance objective} $\big( S, \cF, F^*, c, \ell, u, \mathrm{H}_{\infty,w}(\cdot) \big)$, we seek a \emph{deviation vector} $p \in \mathbb{R}^S$ such that 
\begin{enumerate}[label=(\alph*)]\itemsep0em
\item \label{it:a} $F^*$ is a minimum cost member of $\cF$ with respect to $c-p$,
\item \label{it:b} $p$ is within the bounds $\ell \leq p\leq u$, and
\item \label{it:c} $\mathrm{H}_{\infty,w}(p) \coloneqq \max \left\{ w(s)\mid s \in S, \, p(s) \ne 0 \right\}$ is minimized.
\end{enumerate}
In the \emph{constrained minimum-cost inverse optimization problem under weighted $\ell_\infty$-norm objective} $( S, \cF, F^*, c, \ell, u, \| \cdot \|_{\infty, w})$, condition~\ref{it:c} modifies to
\begin{enumerate}[label=(\alph*')]\itemsep0em\setcounter{enumi}{2}
\item \label{it:c'} $\|p\|_{\infty, w} \coloneqq \max \left\{ w(s) \cdot |p(s)| \bigm| s \in S \right\}$
\end{enumerate}
Due to the lower and upper bounds $\ell$ and $u$, it might happen that there exists no deviation vector $p$ satisfying the requirements. A deviation vector is called \emph{feasible} if it satisfies conditions~\ref{it:a} and~\ref{it:b}, and \emph{optimal} if in addition it attains the minimum in~\ref{it:c} or~\ref{it:c'}. We denote the problems by $\big( S, \cF, F^*, c, -\infty, +\infty, \mathrm{H}_{\infty,w}(\cdot) \big)$ and $( S, \cF, F^*, c, -\infty, +\infty,\|\cdot\|_{\infty, w} )$ when no bounds are given on the coordinates of $p$ at all, and call these problems \emph{unconstrained}. 

As an extension, we also consider \emph{multiple underlying optimization} problems at the same time. In this setting, instead of a single cost function, we are given $k$ cost functions $c^1,\dots,c^k$ together with an input solution $F^*$, and our goal is to find a single deviation vector $p$ with $\ell \leq p \leq u$ such that $F^*$ has minimum cost with respect to $c^j-p$ for all $j \in [k]$. In other words, condition~\ref{it:a} modifies to
\begin{enumerate}[label=(\alph*')]
\item \label{it:a'} $F^*$ is a minimum cost member of $\cF$ with respect to $c^j-p$ for $j\in[k]$.
\end{enumerate}
In case of multiple cost functions, we use $\{c^j\}_{j\in[k]}$ instead of $c$ when denoting the problems.

\paragraph*{Our results.}

Our main results are simple, purely combinatorial algorithms that efficiently solve the above, general problems. For the weighted bottleneck Hamming distance objective, we present an algorithm that makes $O(n)$ calls to the oracle $\cO$. In particular, the algorithm runs in strongly polynomial time, assuming that a strongly polynomial algorithm for the underlying optimization problem is available.

For the weighted $\ei$-norm objective, we give an algorithm for finding an optimal deviation vector that makes $O(n\cdot\|w\|_{-1})$ calls to the oracle $\cO$. In particular, the algorithm runs in strongly polynomial time for unit weights if the oracle $\cO$ for the underlying optimization problem can be realized by a strongly polynomial algorithm. Furthermore, we provide a min-max characterization for the minimum size of an optimal deviation vector in the unconstrained setting, i.e.\ when $\ell\equiv-\infty$ and $u \equiv+\infty$.

For both objectives, we show how to solve the problem when multiple cost functions are given instead of a single one.

The proposed algorithms do not rely on the standard techniques commonly used in the literature, i.e.\ binary search and LP-based methods. Instead, we suggest a Newton-type algorithm that iteratively updates the cost function, resembling the approach of Zhang and Liu~\cite{zhang2002general} for the $\ei$-norm objective. They showed that if the inverse optimization problem can be reformulated as a certain maximization problem using dominant sets, then Radzik's method~\cite{radzik1993parametric} provides a strongly polynomial algorithm for finding an optimal solution. In contrast, our algorithms apply to general inverse optimization problems. Furthermore, we consider the constrained setting in which the coordinates of the deviation vector are ought to fall within given lower and upper bounds, hence the cost function has to be updated carefully. For these reasons, Radzik's method cannot be applied to get a strongly polynomial algorithm. 

A high-level description of the algorithm is given by the following scheme.
\begin{symenum}\itemsep0em
\itemsymbol{Step~1.} Choose $p_0$ minimizing the objective such that $\ell\leq p_0\leq u$, set $c_0\coloneqq c-p_0$ and $i \coloneqq 0$.
\itemsymbol{Step~2.} Let $F_i$ be an optimal solution of the underlying optimization problem with respect to~$c_i$.
\itemsymbol{Step~3.} If $c_i(F^*)=c_i(F_i)$, then $p_i$ is an optimal deviation vector and stop. Otherwise, find $p_{i+1}$ satisfying $\ell\leq p_{i+1}\leq u$ and $(c-p_{i+1})(F^*)=(c-p_{i+1})(F_i)$, and minimizing the objective. If no such $p_{i+1}$ exists, then the problem is infeasible and stop. Otherwise set $i \gets i+1$ and go back to Step~2.
\end{symenum}

\medskip

The rest of the paper is organized as follows. Section~\ref{sec:bottle} presents a strongly polynomial algorithm for the weighted bottleneck Hamming distance objective, including the case of multiple cost functions. The weighted $\ei$-norm objective is discussed in Section~\ref{sec:infty}, where first we provide a min-max characterization for the weighted $\ei$-norm of an optimal deviation vector in the unconstrained setting, then give an algorithm for the constrained setting, including the case of multiple cost functions.

\section{Weighted bottleneck Hamming distance objective}
\label{sec:bottle}

As an illustration of our technique, we first consider the problem of minimizing the weighted bottleneck Hamming distance of the original and the modified cost functions. In Section~\ref{sec:spec_ham}, we show that there exists an optimal deviation vector having a restricted structure. We characterize the feasibility of the problem in Section~\ref{sec:feas_ham}. The algorithm for the case of a single cost function is presented in Section~\ref{sec:alg_ham}. We explain how to extend the algorithm for multiple cost functions in Section~\ref{sec:multi_ham}.

\subsection{Optimal deviation vectors}
\label{sec:spec_ham}

Consider an instance $\big( S, \cF, F^*, c, \ell, u, \mathrm{H}_{\infty,w}(\cdot) \big)$ of the constrained minimum-cost inverse optimization problem under the weighted bottleneck Hamming distance objective, where ${w\in\mathbb{R}^S_+}$ is a positive weight function. For ease of discussion, we define
\begin{equation*}
m \coloneqq \max\left\{0, ~ \max_{F \in \cF} \left( c(F^*) - c(F) - \sum_{\substack{s \in F^*\setminus F \\ u(s) < 0}} u(s) + \sum_{\substack{s \in F\setminus F^* \\ \ell(s) > 0}} \ell(s) \right)\right\} .
\end{equation*}
Recall that $\cO$ denotes an algorithm that determines an optimal solution of the underlying optimization problem $(S, \cF, c')$ for any cost function $c'$. Observe that if $\cO$ runs in strongly polynomial time, then the value of $m$ can be determined in strongly polynomial time. For any $\delta\geq 0$, let $\p{\delta}{\ell, u}{w}\colon S\to\mathbb{R}$ be defined as 
\begin{equation*}
\p{\delta}{\ell, u}{w}(s)\coloneqq 
\begin{cases}
u(s) & \text{if $s \in F^*$, $u(s) \ne + \infty$ and $w(s) \le \delta$}, \\
m & \text{if $s \in F^*$, $u(s) = + \infty$ and $w(s) \le \delta$}, \\
\ell(s) & \text{if $s \in S\setminus F^*$, $\ell(s) \ne - \infty$ and $w(s) \le \delta$}, \\
-m & \text{if $s \in S\setminus F^*$, $\ell(s) = - \infty$ and $w(s) \le \delta$}, \\
0 & \text{otherwise}.
\end{cases}
\end{equation*}

The following lemma shows that there exists an optimal deviation vector of special form.

\begin{lemma} \label{lem:inv_min_cost_bottleneck_delta}
Let $\big( S, \cF, F^*, c, \ell, u, \mathrm{H}_{\infty, w}(\cdot) \big)$ be a feasible minimum-cost inverse optimization problem and let $p$ be an optimal deviation vector. Then $\p{\delta}{\ell, u}{w}$ is also an optimal deviation vector, where $\delta \coloneqq \mathrm{H}_{\infty, w}(p)$.
\end{lemma}
\begin{proof}
The lower and upper bounds $\ell\leq\p{\delta}{\ell,u}{w}\leq u$ hold by definition, hence \ref{it:b} is satisfied.

Now we show that \ref{it:a} holds. Let $F \in \cF$ be an arbitrary solution. Then
\begin{align*}
&(c - \p{\delta}{\ell, u}{w})(F^*) - (c - \p{\delta}{\ell, u}{w})(F)\\[2pt]
{}&{}~~=
\left( c(F^*) - \sum_{s \in F^*} \p{\delta}{\ell, u}{w}(s) \right) - \left( c(F) - \sum_{s \in F} \p{\delta}{\ell, u}{w}(s) \right) \\[2pt]
{}&{}~~=
c(F^*) - c(F) - \sum_{s \in F^*\setminus F} \p{\delta}{\ell, u}{w}(s) + \sum_{s \in F\setminus F^*} \p{\delta}{\ell, u}{w}(s) \\[2pt]
{}&{}~~=
c(F^*) - c(F) - \sum_{\substack{s \in F^* \setminus F \\ u(s) \ne + \infty \\ w(s) \le \delta}} u(s) - \sum_{\substack{s \in F^*\setminus F \\ u(s) = + \infty \\ w(s) \le \delta}} m + \sum_{\substack{s \in F \setminus F^* \\ \ell(s) \ne - \infty \\ w(s) \le \delta}} \ell(s) + \sum_{\substack{s \in F \setminus F^* \\ \ell(s) = - \infty \\ w(s) \le \delta}} (-m) .
\end{align*}
If $\big\{ s \in F^*\setminus F \bigm| u(s) = + \infty,\ w(s)\leq\delta \big\} \cup \big\{ s \in F \setminus F^* \bigm| \ell(s) = - \infty,\ w(s)\leq\delta \big\} = \emptyset$, then
\begin{align*}
&(c - \p{\delta}{\ell, u}{w})(F^*) - (c - \p{\delta}{\ell, u}{w})(F)\\[2pt] 
{}&{}~~= 
c(F^*) - c(F) - \sum_{\substack{s \in F^*\setminus F \\ u(s) \ne + \infty \\ w(s) \le \delta}} u(s) + \sum_{\substack{s \in F \setminus F^* \\ \ell(s) \ne - \infty \\ w(s) \le \delta}} \ell(s) \\[2pt]
{}&{}~~\le 
c(F^*) - c(F) - \sum_{\substack{s \in F^*\setminus F \\ u(s) \ne + \infty \\ w(s) \le \delta}} p(s) + \sum_{\substack{s \in F \setminus F^* \\ \ell(s) \ne - \infty \\ w(s) \le \delta}} p(s)\\[2pt]
{}&{}~~= 
(c-p)(F^*) - (c-p)(F)\\[2pt]
{}&{}~~\le 
0.
\end{align*}

Otherwise $\big\{ s \in F^*\setminus F \bigm| u(s) = + \infty,\ w(s)\leq\delta \big\} \cup \big\{ s \in F \setminus F^* \bigm| \ell(s) = - \infty,\ w(s)\leq\delta \big\} \neq \emptyset$. Note that, by the feasibility of $p$ and by the definition of $\delta$, we have $\ell(s) \le 0 \le u(s)$ whenever $w(s) > \delta$. Thus we obtain
\begin{align*}
&(c - \p{\delta}{\ell, u}{w})(F^*) - (c - \p{\delta}{\ell, u}{w})(F)\\[2pt]
{}&{}~~=
c(F^*) - c(F) - \sum_{\substack{s \in F^*\setminus F \\ u(s) \ne + \infty \\ w(s) \le \delta}} u(s) - \sum_{\substack{s \in F^*\setminus F \\ u(s) = + \infty \\ w(s) \le \delta}} m
+ \sum_{\substack{s \in F \setminus F^* \\ \ell(s) \ne - \infty \\ w(s) \le \delta}} \ell(s) + \sum_{\substack{s \in F \setminus F^* \\ \ell(s) = - \infty \\ w(s) \le \delta}} (-m) \\[2pt]
{}&{}~~\le 
c(F^*) - c(F) - \sum_{\substack{s \in F^*\setminus F \\ u(s) \ne + \infty \\ w(s) \le \delta}} u(s) + \sum_{\substack{s \in F \setminus F^* \\ \ell(s) \ne - \infty \\ w(s) \le \delta}} \ell(s) - 1 \cdot m \\[2pt]
{}&{}~~\le
c(F^*) - c(F) - \sum_{\substack{s \in F^*\setminus F \\ u(s) < 0 \\ w(s) \le \delta}} u(s) + \sum_{\substack{s \in F \setminus F^* \\ \ell(s) > 0 \\ w(s) \le \delta}} \ell(s) - m \\[2pt]
{}&{}~~=
c(F^*) - c(F) - \sum_{\substack{s \in F^*\setminus F \\ u(s) < 0}} u(s) + \sum_{\substack{s \in F \setminus F^* \\ \ell(s) > 0}} \ell(s) - m\\[2pt]
{}&{}~~\le
0,
\end{align*}
where the last inequality holds by the definition of $m$. Therefore, \ref{it:a} is indeed satisfied.

Finally, to see \ref{it:c}, observe that $\mathrm{H}_{\infty, w}(\p{\delta}{\ell, u}{w}) \le \delta = \mathrm{H}_{\infty, w}(p)$, hence $\p{\delta}{\ell, u}{w}$ is also optimal.
\end{proof}

By Lemma~\ref{lem:inv_min_cost_bottleneck_delta}, it suffices to look for the optimal deviation vector among vectors of special form. It turns out that the value of $\delta$ can be chosen from the values of the weight function $w$.

\begin{lemma} \label{lem:inv_min_cost_bottleneck_observation_on_delta}
 Let $\big( S, \cF, F^*, c, \ell, u, \mathrm{H}_{\infty, w}(\cdot) \big)$ be a feasible minimum-cost inverse optimization problem and let $\delta \ge 0$ be such that $\p{\delta}{\ell, u}{w}$ is a feasible deviation vector. Then the following hold.
 \begin{enumerate}[label=(\roman*)]\itemsep0em
  \item \label{bott_obs:a} There exists $s \in S$ with $w(s) \le \delta$ for which $\p{w(s)}{\ell, u}{w}$ is also a feasible deviation vector.
  \item \label{bott_obs:b} For any $\delta' \ge \delta$, the deviation vector $\p{\delta'}{\ell, u}{w}$ is also feasible.
 \end{enumerate}
\end{lemma}
\begin{proof}
To see \ref{bott_obs:a}, let $s \in S$ be such an element for which $w(s) = \max \{ w(s') \mid s' \in S, \, w(s') \le \delta \}$ holds. Then by definition, we have $\p{\delta}{\ell, u}{w} = \p{w(s)}{\ell, u}{w}$.

For \ref{bott_obs:b}, let $F \in \mathcal{F}$ be arbitrary. Note that, by the feasibility of $\p{\delta}{\ell, u}{w}$, we have $\ell(s) \le 0 \le u(s)$ whenever $w(s) > \delta$. Then
\begin{align*}
&(c - \p{\delta'}{\ell, u}{w})(F^*) - (c - \p{\delta'}{\ell, u}{w})(F)\\[2pt]
{}&{}~~=
c(F^*) - c(F) - \sum_{\substack{s \in F^*\setminus F \\ u(s) \ne + \infty \\ w(s) \le \delta'}} u(s) - \sum_{\substack{s \in F^*\setminus F \\ u(s) = + \infty \\ w(s) \le \delta'}} m
+ \sum_{\substack{s \in F \setminus F^* \\ \ell(s) \ne - \infty \\ w(s) \le \delta'}} \ell(s) + \sum_{\substack{s \in F \setminus F^* \\ \ell(s) = - \infty \\ w(s) \le \delta'}} (-m) \\[2pt]
{}&{}~~\le
c(F^*) - c(F) - \sum_{\substack{s \in F^*\setminus F \\ u(s) \ne + \infty \\ w(s) \le \delta}} u(s) - \sum_{\substack{s \in F^*\setminus F \\ u(s) = + \infty \\ w(s) \le \delta}} m
+ \sum_{\substack{s \in F \setminus F^* \\ \ell(s) \ne - \infty \\ w(s) \le \delta}} \ell(s) + \sum_{\substack{s \in F \setminus F^* \\ \ell(s) = - \infty \\ w(s) \le \delta}} (-m) \\[2pt]
{}&{}~~=
(c - \p{\delta}{\ell, u}{w})(F^*) - (c - \p{\delta}{\ell, u}{w})(F)\\[2pt]
{}&{}~~\le
0,
\end{align*}
concluding the proof of the lemma.
\end{proof}

\subsection{Characterizing feasibility}
\label{sec:feas_ham}

We give a necessary and sufficient condition for the feasibility of the minimum-cost inverse optimization problem $\big( S, \cF, F^*, c, \ell,\allowbreak u, \mathrm{H}_{\infty,w}(\cdot) \big)$.

\begin{lemma} \label{lem:inv_min_cost_bottleneck_n&s_cond_for_infeas}
Let $\big( S, \cF, F^*, c, \ell, u, \mathrm{H}_{\infty, w}(\cdot) \big)$ be a minimum-cost inverse optimization problem. Then the problem is feasible if and only if $\p{w_{\max}}{\ell, u}{w}$ is a feasible deviation vector, where
$w_{\max} \coloneqq \max \left\{ w(s) \bigm| s \in S \right\}$.
\end{lemma}
\begin{proof}
Clearly, if $\p{w_{\max}}{\ell, u}{w}$ is feasible, then so is the problem.

To see the other direction, suppose to the contrary that $\p{w_{\max}}{\ell, u}{w}$ is not feasible, but there exists a feasible deviation vector $p$. If, in addition, $p$ is chosen to be optimal, then, by Lemma~\ref{lem:inv_min_cost_bottleneck_delta}, the deviation vector $\p{\delta}{\ell, u}{w}$ is also optimal for $\delta \coloneqq \mathrm{H}_{\infty, w}(p)$. Obviously, $\delta \le w_{\max}$ holds. By Lemma~\ref{lem:inv_min_cost_bottleneck_observation_on_delta}, this implies the feasibility of $\p{w_{\max}}{\ell, u}{w}$, a contradiction.
\end{proof}

\subsection{Algorithm}
\label{sec:alg_ham}

We turn to the description of the algorithm and its analysis. The high-level idea is as described in the introduction. In each iteration, we determine an optimal solution $F\in\cF$ using the oracle $\cO$ as a black box. If the cost of $F$ equals that of $F^*$, then we stop. Otherwise, we modify the costs in such a way that $F$ is ``eliminated'', that is, $F$ and $F^*$ share the same cost with respect to the modified cost function -- hence the name \emph{Newton-type}. The algorithm is presented as Algorithm~\ref{algo:inv_min_cost_bottleneck}.

\begin{algorithm}[!ht] 
\caption{Algorithm for the constrained minimum-cost inverse optimization problem under the weighted bottleneck Hamming distance objective}
\label{algo:inv_min_cost_bottleneck}
\DontPrintSemicolon

\KwIn{A minimum-cost inverse optimization problem $(S, \cF , F^*, c, \ell, u, \mathrm{H}_{\infty, w}(\cdot))$ and an oracle $\cO$ for the minimum-cost optimization problem $(S, \cF , c')$ with any cost function $c'$.
}
\KwOut{An optimal deviation vector if the problem is feasible, otherwise \texttt{Infeasible}.}

\smallskip

$\delta_0 \gets \max \big\{ 0,\, \max \{ w(s) \mid s \in S, \ell(s) > 0 \},\,\max \{ w(s) \mid s \in S, u(s) < 0 \} \big\}$\;
$c_0 \gets c - \p{\delta_0}{\ell, u}{w}$\;
$F_0 \gets \text{a minimum $c_0$-cost member of $\cF$ determined by $\cO$}$\;
$i \gets 0$\;
\While{$c_i(F^*) > c_i(F_i)$}{
$S_i \gets \left\{ s \in S \mid w(s) > \delta_i \right\}$\;
\uIf{$S_i = \emptyset$}{
\Return{\tt Infeasible}\;
}
\Else{
$\delta_{i+1} \gets \min \left\{ w(s) \mid s \in S_i \right\}$\;
}
$c_{i+1} \gets c - \p{\delta_{i+1}}{\ell, u}{w}$\;
$F_{i+1} \gets \text{a minimum $c_{i+1}$-cost member of $\cF$ determined by $\cO$}$\;
$i \gets i+1$\;
}
\Return{$\p{\delta_i}{\ell, u}{w}$}\;
\end{algorithm}

It remains to prove correctness and the running time of the algorithm.

\begin{thm} \label{thm:inv_min_cost_bottleneck_correctness_of_algo}
Algorithm~\ref{algo:inv_min_cost_bottleneck} determines an optimal deviation vector, if exists, for the minimum-cost inverse optimization problem $\big( S, \cF, F^*, c, \ell, u, \mathrm{H}_{\infty, w}(\cdot) \big)$ using $O(n)$ calls to the oracle $\cO$.
\end{thm}
\begin{proof}
We discuss the time complexity and the correctness of the algorithm separately.
\medskip

\noindent \emph{Time complexity.} We show that the algorithm terminates after at most $n$ iterations of the while loop. To see this, observe that if $F^*$ is not a minimum $c_i$-cost member of $\cF$ for some $i$, then either $S_{i+1} \subsetneq S_i$ by the definition of $\delta_{i+1}$, or the algorithm declares the problem to be infeasible. As the size of the set $S_i$ can decrease at most $|S|=n$ times, the statement follows.
\medskip

\noindent \emph{Correctness.} By the above, the algorithm terminates after a finite number of iterations. Observe that if the algorithm returns \texttt{Infeasible}, then it correctly recognizes the problem to be infeasible by Lemma~\ref{lem:inv_min_cost_bottleneck_n&s_cond_for_infeas}.

Assume now that the algorithm terminates with returning a deviation vector $\p{\delta_i}{\ell, u}{w}$ whose feasibility follows from the fact that the while loop ended. If $F^*$ is a minimum $c_0$-cost member of $\cF$, then we are clearly done. Otherwise, there exists an index $q$ such that $F^*$ is a minimum $c_{q+1}$-cost member of $\cF$. Suppose to the contrary that $\p{\delta_{q+1}}{\ell, u}{w}$ is not optimal. Since the problem is feasible, by Lemma~\ref{lem:inv_min_cost_bottleneck_delta}, there exists $\delta<\delta_{q+1}$ such that the deviation vector $\p{\delta}{\ell, u}{w}$ is optimal. Note that $\p{\delta_q}{\ell, u}{w}$ is not a feasible deviation vector since $(c - \p{\delta_q}{\ell, u}{w})(F^*) > (c - \p{\delta_q}{\ell, u}{w})(F_q)$. By Lemma~\ref{lem:inv_min_cost_bottleneck_observation_on_delta}, we get $\delta_q < \delta < \delta_{q+1}$. However, by Lemma~\ref{lem:inv_min_cost_bottleneck_observation_on_delta}, we know that $\delta=w(s)$ for some $s\in S$, contradicting the definition of $\delta_{q+1}$.
\end{proof}

Note that the Algorithm~\ref{algo:inv_min_cost_bottleneck} runs in strongly polynomial time assuming that $\cO$ can be realized by a strongly polynomial-time algorithm.

\subsection{Multiple cost functions}
\label{sec:multi_ham}

Consider now an instance $\big( S, \cF, F^*, \{c^j\}_{j\in[k]}, \ell, u, \mathrm{H}_{\infty,w}(\cdot) \big)$ of the problem with multiple cost functions. By Lemma~\ref{lem:inv_min_cost_bottleneck_delta}, for each $j\in[k]$, there exists $\delta^j\geq 0$ such that $\p{\delta^j}{\ell, u}{w}$ is an optimal deviation vector for the problem $\big( S, \cF, F^*, c^j, \ell, u, \mathrm{H}_{\infty, w}(\cdot) \big)$. Let $\delta\coloneqq\max \big\{ \delta^j\mid j\in[k] \big\}$. By Lemma~\ref{lem:inv_min_cost_bottleneck_observation_on_delta}, $\p{\delta}{\ell, u}{w}$ is a feasible deviation vector for the problem $\big( S, \cF, F^*, c^j, \ell, u, \mathrm{H}_{\infty, w}(\cdot) \big)$ for $j\in[k]$, and it is clearly optimal. Therefore, we get the following.

\begin{cor}
 \hspace{-2pt}The minimum-cost inverse optimization problem $\big( S, \cF, F^*, \{ c^j \}_{j \in [k]}, \ell, u, \allowbreak \mathrm{H}_{\infty,w}(\cdot) \big)$ with multiple cost functions can be solved using $O(k\cdot n)$ calls to the oracle $\cO$.
\end{cor}

\section{Weighted \texorpdfstring{$\ei$}{L-infinity}-norm objective}
\label{sec:infty}

Next we consider the weighted $\ei$-norm objective. Similarly to the case of the weighted bottleneck Hamming distance, we first prove that there exists an optimal deviation vector of a special form in Section~\ref{sec:spec_infty}. We characterize the feasibility of the problem in Section~\ref{sec:feas_infty}. Then we present an algorithm for the case of a single cost function in the constrained setting in Section~\ref{sec:alg_infty}. We explain how to extend the algorithm for multiple cost functions in Section~\ref{sec:multi_infty}. Finally, we give a min-max characterization of the weighted $\ei$-norm of an optimal deviation vector in the unconstrained setting with multiple cost functions in Section~\ref{sec:min-max_infty}. 

\subsection{Optimal deviation vectors}
\label{sec:spec_infty}

Consider an instance $( S, \cF, F^*, c, \ell, u, \|\cdot\|_{\infty,w})$ of the constrained minimum-cost inverse optimization problem under the weighted $\ei$-norm objective, where $w\in\mathbb{R}^S_+$ is a positive weight function. For any $\delta\geq 0$, let $\p{\delta}{\ell, u}{w}\colon S\to\mathbb{R}$ be defined as
\begin{equation*}
\p{\delta}{\ell, u}{w}(s)\coloneqq
\begin{cases}
\ell(s) & \text{if $s \in F^*$ and $\delta/w(s) < \ell(s)$}, \\
\delta/w(s) & \text{if $s \in F^*$ and $\ell(s) \le \delta/w(s) \le u(s)$}, \\
u(s) & \text{if $s \in F^*$ and $u(s) < \delta/w(s)$}, \\
\ell(s) & \text{if $s \in S\setminus F^*$ and $- \delta/w(s) < \ell(s)$}, \\
- \delta/w(s) & \text{if $s \in S\setminus F^*$ and $\ell(s) \le - \delta/w(s) \le u(s)$}, \\
u(s) & \text{if $s \in S\setminus F^*$ and $u(s) < - \delta/w(s)$}.
\end{cases}
\end{equation*}
We simply write $\p{\delta}{}{w}$ when $\ell\equiv-\infty$ and $u\equiv+\infty$. The following technical lemma shows that there exists an optimal deviation vector of special form.

\begin{lemma} \label{lem:inv_min_cost_infinity_delta}
Let $\left( S, \cF , F^*, c, \ell, u, \| \cdot \|_{\infty, w} \right)$ be a feasible minimum-cost inverse optimization problem and let $p$ be an optimal deviation vector.  Then $\p{\delta}{\ell, u}{w}$ is also an optimal deviation vector, where $\delta \coloneqq  \max \left\{ w(s) \cdot | p(s) | \bigm| s \in S \right\}$.
\end{lemma}
\begin{proof}
The lower and upper bounds $\ell \le \p{\delta}{\ell, u}{w} \le u$ hold by definition, hence \ref{it:b} is satisfied. 

Now we show that \ref{it:a} holds. The assumption $\ell \le p \le u$ and the definition of $\delta$ imply that $-\delta/w(s) \le p(s) \le u(s)$ and $\ell(s) \le p(s) \le \delta/w(s)$ hold for every $s \in S$. Let $F \in \cF $ be an arbitrary solution. Then
\begin{align*}
&(c - \p{\delta}{\ell, u}{w})(F^*) - (c - \p{\delta}{\ell, u}{w})(F)\\[2pt]
{}&{}~~= 
\left( c(F^*) - \sum_{s \in F^*} \p{\delta}{\ell, u}{w}(s) \right) - \left( c(F) - \sum_{s \in F} \p{\delta}{\ell, u}{w}(s) \right) \\[2pt]
{}&{}~~= 
c(F^*) - c(F) - \sum_{s \in F^*\setminus  F} \p{\delta}{\ell, u}{w}(s) + \sum_{s \in F\setminus  F^*} \p{\delta}{\ell, u}{w}(s) \\[2pt]
{}&{}~~= 
c(F^*) - c(F) - \sum_{\makemathbox[.8\width]{\substack{s \in F^*\setminus  F \\ \delta/w(s)\geq \ell(s)\\ \delta/w(s) \le u(s)}}} \frac{\delta}{w(s)} - \sum_{\makemathbox[.8\width]{\substack{s \in F^*\setminus  F \\ u(s) < \delta/w(s)}}} u(s) + \sum_{\makemathbox[.8\width]{\substack{s \in F\setminus  F^* \\ - \delta/w(s) < \ell(s)}}} \ell(s) + \sum_{\makemathbox[.8\width]{\substack{s \in F\setminus  F^* \\ -\delta/w(s)\geq \ell(s)\\ - \delta/w(s) \le u(s)}}} \left( - \, \frac{\delta}{w(s)} \right) \\[2pt]
{}&{}~~\le 
c(F^*) - c(F) - \sum_{s \in F^*\setminus  F} p(s) + \sum_{s \in F\setminus  F^*} p(s) \\[2pt]
{}&{}~~= 
(c-p)(F^*) - (c-p)(F) \\[2pt]
{}&{}~~\le
0,
\end{align*}
where the last inequality holds by the feasibility of $p$. 

Finally, to see that \ref{it:c} holds for $\p{\delta}{\ell, u}{w}$, observe that $\| \p{\delta}{\ell, u}{w} \|_{\infty, w} \le \delta = \| p \|_{\infty, w}$. This concludes the proof of the lemma.
\end{proof}

By Lemma~\ref{lem:inv_min_cost_infinity_delta}, it suffices to look for the optimal deviation vector among vectors of special form. Furthermore, we get the following useful property of deviation vectors of such form.

\begin{lemma} \label{lem:inv_min_cost_infinity_observation_on_delta}
 Let $\left( S, \cF, F^*, c, \ell, u, \| \cdot \|_{\infty, w} \right)$ be a feasible minimum-cost inverse optimization problem and let $\delta \ge 0$ be such that $\p{\delta}{\ell, u}{w}$ is a feasible deviation vector. Then for any $\delta' \ge \delta$, the deviation vector $\p{\delta'}{\ell, u}{w}$ is also feasible.
\end{lemma}
\begin{proof}
Let $F \in \cF $ be an arbitrary solution. Then 
\begin{align*}
&(c - \p{\delta'}{\ell, u}{w})(F^*) - (c - \p{\delta'}{\ell, u}{w})(F)\\[2pt]
{}&{}~~=
c(F^*) - c(F) - \sum_{\makemathbox[.8\width]{\substack{s \in F^*\setminus  F \\ \delta'/w(s) < \ell(s)}}} \ell(s) - \sum_{\makemathbox[.8\width]{\substack{s \in F^*\setminus  F \\ \delta'/w(s)\geq \ell(s)\\ \delta'/w(s) \le u(s)}}} \frac{\delta'}{w(s)} - \sum_{\makemathbox[.8\width]{\substack{s \in F^*\setminus  F \\ u(s) < \delta'/w(s)}}} u(s) \\[2pt]
{}&{}~~~~
+ \sum_{\makemathbox[.8\width]{\substack{s \in F\setminus  F^* \\ - \delta'/w(s) < \ell(s)}}} \ell(s) + \sum_{\makemathbox[.8\width]{\substack{s \in F\setminus  F^* \\ -\delta'/w(s)\geq \ell(s)\\ - \delta'/w(s) \le u(s)}}} \left( - \, \frac{\delta'}{w(s)} \right) + \sum_{\makemathbox[.8\width]{\substack{s \in F\setminus  F^* \\ u(s) < - \delta'/w(s)}}} \ell(s) \\[2pt]
{}&{}~~\le
c(F^*) - c(F) - \sum_{\makemathbox[.8\width]{\substack{s \in F^*\setminus  F \\ \delta/w(s) < \ell(s)}}} \ell(s) - \sum_{\makemathbox[.8\width]{\substack{s \in F^*\setminus  F \\ \delta/w(s)\geq \ell(s)\\ \delta'/w(s) \le u(s)}}} \frac{\delta}{w(s)} - \sum_{\makemathbox[.8\width]{\substack{s \in F^*\setminus  F \\ u(s) < \delta/w(s)}}} u(s) \\[2pt]
{}&{}~~~~
+ \sum_{\makemathbox[.8\width]{\substack{s \in F\setminus  F^* \\ - \delta/w(s) < \ell(s)}}} \ell(s) + \sum_{\makemathbox[.8\width]{\substack{s \in F\setminus  F^* \\ -\delta/w(s)\geq \ell(s)\\ - \delta/w(s) \le u(s)}}} \left( - \frac{\delta}{w(s)} \right) + \sum_{\makemathbox[.8\width]{\substack{s \in F\setminus  F^* \\ u(s) < - \delta/w(s)}}} \ell(s) \\[2pt]
{}&{}~~=
(c - \p{\delta}{\ell, u}{w})(F^*) - (c - \p{\delta}{\ell, u}{w})(F)\\[2pt]
{}&{}~~\le
0,
\end{align*}
concluding the proof of the lemma.
\end{proof}

\subsection{Characterizing feasibility}
\label{sec:feas_infty}

We give a necessary and sufficient condition for the feasibility of the minimum-cost inverse optimization problem $\left( S, \cF , F^*, c, \ell, u, \| \cdot \|_{\infty, w} \right)$. 

\begin{lemma} \label{lem:inv_min_cost_infinity_n&s_cond_for_infeas}
Let $\left( S, \cF , F^*, c, \ell, u, \| \cdot \|_{\infty, w} \right)$ be a minimum-cost inverse optimization problem. For any $F\in\cF $, define
\begin{equation*}
W(F)\coloneqq  
\begin{cases}
\displaystyle\frac{1}{\displaystyle\sum_{\substack{s \in F^*\setminus  F \\ u(s) = + \infty}} \frac{1}{w(s)} + \sum_{\substack{s \in F\setminus  F^* \\ \ell(s) = - \infty}} \frac{1}{w(s)}} & \text{if the divisor is not $0$}, \\[5pt]
\,0 & \text{otherwise},
\end{cases} 
\end{equation*}
and let
\begin{align*}
m_1{}&{}\coloneqq \max \left\{ w(s) \cdot | u(s) | \bigm| s \in F^*,\, u(s) \ne + \infty \right\},\\
m_2{}&{}\coloneqq \max \left\{ w(s) \cdot | \ell(s) | \bigm| s \in S\setminus F^*,\, \ell(s) \ne - \infty \right\},\\
m_3{}&{}\coloneqq \max_{F \in \cF } \left( W(F) \cdot \left( c(F^*) - c(F) - \sum_{\substack{s \in F^*\setminus  F \\ u(s) \ne + \infty}} u(s) + \sum_{\substack{s \in F\setminus  F^* \\ \ell(s) \ne - \infty}} \ell(s) \right)\right).
\end{align*}
Then the problem is feasible if and only if $\p{m}{\ell, u}{w}$ is a feasible deviation vector for
\[ m\coloneqq\max\{0,m_1,m_2,m_3\} . \]
\end{lemma}
\begin{proof}
Clearly, if $\p{m}{\ell, u}{w}$ is feasible, then so is the problem.

To see the other direction, suppose to the contrary that $\p{m}{\ell, u}{w}$ is not feasible, but there exists a feasible deviation vector $p$. Then there exists $F \in \cF $ such that
\begin{align*}
0 
{}&{}< 
(c - \p{m}{\ell, u}{w})(F^*) - (c - \p{m}{\ell, u}{w})(F) \\[2pt]
{}&{}= 
c(F^*) - c(F) - \sum_{\substack{s \in F^*\setminus  F \\ u(s) \ne + \infty}} u(s) - \sum_{\substack{s \in F^*\setminus  F \\ u(s) = + \infty}} \frac{m}{w(s)} + \sum_{\substack{s \in F\setminus  F^* \\ \ell(s) \ne - \infty}} \ell(s) + \sum_{\substack{s \in F\setminus  F^* \\ \ell(s) = - \infty}} \left( - \, \frac{m}{w(s)} \right) \\[2pt]
{}&{}= 
c(F^*) - c(F) - \sum_{\substack{s \in F^*\setminus  F \\ u(s) \ne + \infty}} u(s) + \sum_{\substack{s \in F\setminus  F^* \\ \ell(s) \ne - \infty}} \ell(s) - m \left( \sum_{\substack{s \in F^*\setminus  F \\ u(s) = + \infty}} \frac{1}{w(s)} + \sum_{\substack{s \in F\setminus  F^* \\ \ell(s) = - \infty}} \frac{1}{w(s)} \right).
\end{align*}

If $\{ s \in F^*\setminus  F \mid u(s) = + \infty \} \cup \{ s \in F\setminus  F^* \mid \ell(s) = - \infty \} = \emptyset$, then we obtain
\begin{align*}
0 
{}&{}< 
c(F^*) - c(F) - \sum_{\substack{s \in F^*\setminus  F \\ u(s) \ne + \infty}} u(s) + \sum_{\substack{s \in F\setminus  F^* \\ \ell(s) \ne - \infty}} \ell(s) - m \cdot 0 \\[2pt]
{}&{}\le 
c(F^*) - c(F) - \sum_{s \in F^*\setminus  F} p(s) + \sum_{s \in F\setminus  F^*} p(s) \\[2pt]
{}&{}= \big( c(F^*) - p(F^*) \big) - \big( c(F) - p(F) \big)\\[2pt]
{}&{}\le 
0,
\end{align*}
where the last inequality holds since $p$ is feasible, leading to a contradiction.

If ${\left\{ s \in F^*\setminus  F \mid u(s) = + \infty \right\}} \allowbreak \cup \left\{ s \in F\setminus  F^* \mid \ell(s) = - \infty \right\} \neq \emptyset$, then we obtain
\begin{equation*}
0 < c(F^*) - c(F) - \sum_{\substack{s \in F^*\setminus  F \\ u(s) \ne + \infty}} u(s) + \sum_{\substack{s \in F\setminus  F^* \\ \ell(s) \ne - \infty}} \ell(s) \, - \frac{m}{W(F)},
\end{equation*}
which contradicts the definition of $m$.
\end{proof}

\subsection{Algorithm}
\label{sec:alg_infty}

We turn to the description of the algorithm and its analysis, which is similar to that of Algorithm~\ref{algo:inv_min_cost_bottleneck}. The algorithm is presented as Algorithm~\ref{algo:inv_min_cost_infinity}. 

\begin{algorithm}[!ht]
\caption{Algorithm for the constrained minimum-cost inverse optimization problem under the weighted $\ell_{\infty}$-norm objective}
\label{algo:inv_min_cost_infinity}
\DontPrintSemicolon

\KwIn{A minimum-cost inverse optimization problem $(S, \cF , F^*, c, \ell, u, \| \cdot \|_{\infty, w})$ and an oracle $\cO$ for the minimum-cost optimization problem $(S, \cF , c')$ with any cost function $c'$.
}
\KwOut{An optimal deviation vector if the problem is feasible, otherwise \texttt{Infeasible}.}

\smallskip

$d_0 \gets \max \Big\{ 0, \, \max \big\{ w(s) \cdot \ell(s) \bigm| s \in S,\, \ell(s) > 0 \big\}, \, \max \big\{ w(s) \cdot |u(s)| \bigm| s \in S,\, u(s) < 0 \big\} \Big\}$\;
$c_0 \gets c - \p{d_0}{\ell, u}{w}$\;
$F_0 \gets \text{a minimum $c_0$-cost member of $\cF $ determined by $\cO$}$\;
$i \gets 0$\;
\While{$c_i(F^*) > c_i(F_i)$}{
$S_i \gets \big\{ s \in F^* \bigm| d_i < w(s) \cdot u(s) \big\} \cup \big\{ s \in S\setminus  F^* \bigm| d_i > w(s) \cdot \ell(s) \big\}$\;
\uIf{$(F^* \triangle F_i) \cap S_i \ne \emptyset$}{
${\delta_{i+1} \! \gets \! \min \left\{\displaystyle\frac{c_i(F^*) - c_i(F_i)}{\textstyle \frac{1}{w} \big( (F^* \triangle F_i) \cap S_i \big)},\, \min_{s \in F^* \cap S_i} \! \left\{ u(s) - \frac{d_i}{w(s)}\right\},\, \min_{s \in S_i\setminus F^*} \! \left\{ \frac{d_i}{w(s)} - \ell(s)\right\}\right\}}$\;
}
\Else{
\Return{\tt Infeasible}\;
}
$d_{i+1} \gets d_i + \delta_{i+1}$\;
$c_{i+1} \gets c - \p{d_{i+1}}{\ell, u}{w}$\;
$F_{i+1} \gets \text{a minimum $c_{i+1}$-cost member of $\cF $ determined by $\cO$}$\;
$i \gets i+1$\;
}
\Return{$\p{d_i}{\ell, u}{w}$}\;
\end{algorithm}

For proving the correctness and the running time of the algorithm, we need the following lemmas.

\begin{lemma} \label{lemma:inv_min_cost_infinity_nonneg_of_delta}
If $F^*$ is not a minimum $c_i$-cost member of $\cF $ for some $i$, then either $\delta_{i+1} > 0$ and $S_{i+1}\subseteq S_i$, or Algorithm~\ref{algo:inv_min_cost_infinity} declares the problem to be infeasible.
\end{lemma}
\begin{proof}
The statement follows from the definition of $\delta_{i+1}$ and from Lemma~\ref{lemma:inv_min_cost_infinity_nonneg_of_delta}.
\end{proof}

\begin{lemma} \label{lemma:inv_min_cost_infinity_correction}
If $F^*$ is not a minimum $c_i$-cost member of $\cF $ for some $i$, then $c_{i+1}(F^*) = c_{i+1}(F_i)$, or $ S_{i+1} \subsetneq S_i$, or Algorithm~\ref{algo:inv_min_cost_infinity} declares the problem to be infeasible.
\end{lemma}
\begin{proof}
Let $i$ be an index such that $F^*$ is not a minimum $c_i$-cost member of $\cF$, and assume that Algorithm~\ref{algo:inv_min_cost_infinity} does not declare the problem to be infeasible in the $i$th step. Then $S_{i+1} \subseteq S_i$ holds by Lemma~\ref{lemma:inv_min_cost_infinity_nonneg_of_delta}. If $S_{i+1} \subsetneq S_i$, then we are done, hence consider the case $S_{i+1} = S_i$. Then
\begin{equation*} 
\delta_{i+1} = \frac{c_i(F^*) - c_i(F_i)}{\textstyle \frac{1}{w} \big( (F^* \triangle F_i) \cap S_i \big)},
\end{equation*}
hence we get
\begin{equation*}
c_{i+1}(F^*) - c_{i+1}(F_i) = c_i(F^*) - c_i(F_i) - \sum_{s \in (F^* \triangle F_i) \cap S_i} \frac{\delta_{i+1}}{w(s)} = 0,
\end{equation*}
concluding the proof of the lemma.
\end{proof}

\begin{lemma} \label{lemma:inv_min_cost_infinity_for_poly_time}
If $F^*$ is not a minimum $c_i$-cost member of $\cF $ for some $i$, then
\begin{equation*}
\tfrac{1}{w} \big( (F_i\setminus F^*) \cap S_i \big) - \tfrac{1}{w} \big( (F_i \cap F^*) \cap S_i \big) > \tfrac{1}{w} \big( (F_{i+1}\setminus F^*) \cap S_{i+1} \big) - \tfrac{1}{w} \big( (F_{i+1} \cap F^*) \cap S_{i+1} \big)     
\end{equation*}
or $S_{i+1} \subsetneq S_i$, or Algorithm~\ref{algo:inv_min_cost_infinity} declares the problem to be infeasible.
\end{lemma}
\begin{proof}
Let $i$ be an index such that $F^*$ is not a minimum $c_i$-cost member of $\cF$, and assume that Algorithm~\ref{algo:inv_min_cost_infinity} does not declare the problem to be infeasible in the $i$-th step and that $S_{i+1}=S_i$. Then $c_{i+1}(F^*)=c_{i+1}(F_i)$ hold by Lemma~\ref{lemma:inv_min_cost_infinity_correction}, respectively. Thus we get
\begin{align*}
0 
{}&{}< 
c_{i+1}(F^*) - c_{i+1}(F_{i+1})\\[2pt]
{}&{}= 
c_{i+1}(F_i) - c_{i+1}(F_{i+1}) \\[2pt]
{}&{}= 
\left( c_i(F_i) - \sum_{\makemathbox[.85\width]{s \in (F_i \cap F^*) \cap S_i}} \frac{\delta_{i+1}}{w(s)} + \sum_{\makemathbox[.85\width]{s \in (F_i\setminus F^*) \cap S_i}} \frac{\delta_{i+1}}{w(s)} \right)\\[2pt]
{}&{}~~- \left( c_i(F_{i+1}) - \sum_{\makemathbox[.85\width]{s \in (F_{i+1} \cap F^*) \cap S_i}} \frac{\delta_{i+1}}{w(s)} + \sum_{\makemathbox[.85\width]{s \in (F_{i+1}\setminus F^*) \cap S_i}} \frac{\delta_{i+1}}{w(s)} \right) \\[2pt]
{}&{}= 
c_i(F_i) - c_i(F_{i+1}) + \delta_{i+1} \cdot \Big[ \tfrac{1}{w} \big( (F_i\setminus F^*) \cap S_i \big) - \tfrac{1}{w} \big( (F_i \cap F^*) \cap S_i \big) \Big]\\[2pt]
{}&{}~~~~- \delta_{i+1} \cdot \Big[ \tfrac{1}{w} \big( (F_{i+1}\setminus F^*) \cap S_{i+1} \big) - \tfrac{1}{w} \big( (F_{i+1} \cap F^*) \cap S_{i+1} \big) \Big].
\end{align*}
Since $\delta_{i+1}>0$ by Lemma~\ref{lemma:inv_min_cost_infinity_nonneg_of_delta} and $c_i(F_i)-c_i(F_{i+1})\leq 0$ by the optimality of $F_i$ with respect to~$c_i$, the statement follows.
\end{proof}

With the help of Lemmas~\ref{lemma:inv_min_cost_infinity_nonneg_of_delta}--\ref{lemma:inv_min_cost_infinity_for_poly_time}, we are ready to prove the main result of this section. 

\begin{thm} \label{thm:inv_min_cost_infinity_correctness_of_algo}
Algorithm~\ref{algo:inv_min_cost_infinity} determines an optimal deviation vector, if exists, for the minimum-cost inverse optimization problem $(S,\cF, F^*, c, \ell, u, \| \cdot \|_{\infty, w})$ using $O(n\cdot \|w\|_{-1})$ calls to the $\cO$.
\end{thm}
\begin{proof}
We discuss the time complexity and the correctness of the algorithm separately.
\medskip

\noindent \emph{Time complexity.} Recall that $w\in\mathbb{R}^S_+$ is scaled so that $\frac{1}{w}(X)$ is an integer for each $X\subseteq S$. Between two iterations of the while loop, the size of the set $S_i$ or the value of $\frac{1}{w} \big( (F_i\setminus F^*) \cap S_i \big) - \frac{1}{w} \big( (F_i \cap F^*) \cap S_i \big)$ strictly decreases by Lemma~\ref{lemma:inv_min_cost_infinity_for_poly_time}. The size of $S_i$ can decrease at most $n$ times. Between two iterations where the size of $S_i$ decreases, the value of $\frac{1}{w} \big( (F_i\setminus F^*) \cap S_i \big) - \frac{1}{w} \big( (F_i \cap F^*) \cap S_i \big)$ can decrease at most $2\|w\|_{-1}$ times. Hence the total number of iterations is $O(n\cdot \|w\|_{-1})$.
\medskip

\noindent \emph{Correctness.} By the above, the procedure terminates after a finite number of iterations. First, we show that if the the algorithm returns \texttt{Infeasible}, then it correctly recognizes the problem to be infeasible. To see this, assume that the algorithm terminated in the $i$th step and declared the problem to be infeasible. Then $(F^* \triangle F_i) \cap S_i = \emptyset$, so by the definitions of $S_i$ and $m$ as in Lemma~\ref{lem:inv_min_cost_infinity_n&s_cond_for_infeas}, we obtain
\begin{align*}
&(c - \p{m}{\ell,u}{w})(F^*) - (c - \p{m}{\ell,u}{w})(F_i) \\[2pt]
{}&{}~~= 
c(F^*) - c(F_i) - \sum_{\substack{s \in F^*\setminus F_i \\ u(s) \ne + \infty}} u(s) - \sum_{\substack{s \in F^*\setminus F \\ u(s) = + \infty}} \frac{m}{w(s)} + \sum_{\substack{s \in F\setminus F^* \\ \ell(s) \ne - \infty}} \ell(s) + \sum_{\substack{s \in F\setminus F^* \\ \ell(s) = - \infty}} \left( - \, \frac{m}{w(s)} \right) \\[2pt]
{}&{}~~= 
c(F) - c(F_i) - \sum_{s \in F^*\setminus F_i} u(s) + \sum_{s \in F_i\setminus F^*} \ell(s)\\[2pt]
{}&{}~~= 
(c - \p{d_i}{\ell, u}{w})(F^*) - (c - \p{d_i}{\ell, u}{w})(F_i)\\[2pt]
{}&{}~~> 
0,
\end{align*}
hence the problem is infeasible by Lemma~\ref{lem:inv_min_cost_infinity_n&s_cond_for_infeas}.

Assume now that the algorithm terminates with returning a deviation vector whose feasibility follows from the fact that the while loop ended. If $F^*$ is a minimum $c_0$-cost member of $\cF$, then we are clearly done. Otherwise, there exists an index $q$ such that $F^*$ is a minimum $c_{q+1}$-cost member of $\cF$. Suppose to the contrary that $\p{d_{q+1}}{\ell, u}{w}$ is not optimal. By Lemma~\ref{lem:inv_min_cost_infinity_delta}, there exists $\delta<d_{q+1}$ such that the deviation vector $\p{\delta}{\ell, u}{w}$ is optimal. Thus, by Lemma~\ref{lemma:inv_min_cost_infinity_correction}, we get
\begin{align*}
0
{}&{} \ge 
(c - \p{\delta}{\ell, u}{w})(F^*) - (c - \p{\delta}{\ell, u}{w})(F_q) \\[2pt]
{}&{}= 
c(F^*) - c(F_q) - \sum_{\makemathbox[.8\width]{\substack{s \in F^*\setminus F_q \\ \delta/w(s) \le u(s)}}} \frac{\delta}{w(s)} - \sum_{\makemathbox[.8\width]{\substack{s \in F^*\setminus F_q \\ u(s) < \delta/w(s)}}} u(s) + \sum_{\makemathbox[.8\width]{\substack{s \in F_q\setminus F^* \\ - \delta/w(s) < \ell(s)}}} \ell(s) + \sum_{\makemathbox[.8\width]{\substack{s \in F_q\setminus F^* \\ \ell(s) \le - \delta/w(s)}}} \left( - \, \frac{\delta}{w(s)} \right) \\[2pt]
{}&{}> 
c(F^*) - c(F_q) - \hspace{-0.2cm}\sum_{\makemathbox[.8\width]{\substack{s \in F^*\setminus F_q \\ d_{q+1}/w(s) \le u(s)}}} \frac{d_{q+1}}{w(s)} - \hspace{-0.2cm}\sum_{\makemathbox[.8\width]{\substack{s \in F^*\setminus F_q \\ u(s) < d_{q+1}/w(s)}}} u(s) + \hspace{-0.2cm}\sum_{\makemathbox[.8\width]{\substack{s \in F_q\setminus F^* \\ - d_{q+1}/w(s) < \ell(s)}}} \ell(s) + \hspace{-0.2cm}\sum_{\makemathbox[.8\width]{\substack{s \in F_q \setminus F^* \\ \ell(s) \le - d_{q+1}/w(s)}}} \left( - \, \frac{d_{q+1}}{w(s)} \right) \\[2pt]
{}&{}= 
c_{q+1}(F^*) - c_{q+1}(F_q)\\[2pt]
{}&{}=
0,
\end{align*}
a contradiction. This concludes the proof of the theorem.
\end{proof}

Note that the Algorithm~\ref{algo:inv_min_cost_infinity} runs in strongly polynomial time assuming that $\cO$ can be realized by a strongly polynomial-time algorithm.

\subsection{Multiple cost functions}
\label{sec:multi_infty}

Consider now an instance $\big( S, \cF, F^*, \{c^j\}_{j\in[k]}, \ell, u, \|\cdot\|_{\infty,w} \big)$ of the problem with multiple cost functions. By Lemma~\ref{lem:inv_min_cost_infinity_delta}, for each $j\in[k]$, there exists $\delta^j\geq 0$ such that $\p{\delta^j}{\ell, u}{w}$ is an optimal deviation vector for the problem $\big( S, \cF, F^*, c^j, \ell, u, \|\cdot\|_{\infty, w} \big)$. Let $\delta\coloneqq\max \big\{ \delta^j\mid j\in[k] \big\}$. By Lemma~\ref{lem:inv_min_cost_infinity_observation_on_delta}, $\p{\delta}{\ell, u}{w}$ is a feasible deviation vector for the problem $( S, \cF, F^*, c^j, \ell, u, \|\cdot\|_{\infty, w})$ for $j\in[k]$, and it is clearly optimal. Therefore, we get the following.

\begin{cor} 
The minimum-cost inverse optimization problem\hspace{1pt} $\big( S, \, \cF, \, F^*, \, \{c^{\, j} \}_{j \, \in \, [k]}, \, \ell, \, u, \allowbreak {\| \cdot\|_{\infty,w}} \big)$ with multiple cost functions can be solved using $O(k\cdot n\cdot \|w\|_{-1})$ calls to oracle $\cO$.
\end{cor}

\subsection{Min-max theorem}
\label{sec:min-max_infty}

With the help of Lemmas~\ref{lem:inv_min_cost_infinity_delta} and~\ref{lem:inv_min_cost_infinity_observation_on_delta}, we provide a min-max characterization for the weighted infinity norm of an optimal deviation vector in the unconstrained setting, even for the case of multiple cost functions. Recall that we use the notation $\frac{1}{w}(X)\coloneqq \sum \left\{ \frac{1}{w(s)} \, \middle| \, s\in X \right\}$.

\begin{thm} \label{thm:inv_min_cost_infinity_minmax}
Let $\big( S, \cF, F^*, \{ c^j \}_{j \in [k]},-\infty,+\infty, \| \cdot \|_{\infty, w} \big)$ be a feasible minimum-cost inverse optimization problem with multiple cost functions. Then
\begin{gather*}
\min \big\{\| p \|_{\infty, w} \bigm| \text{$p$ is a feasible deviation vector} \big\} \\
= \max \left\{0, \, \max \left\{\frac{c^j(F^*) - c^j(F)}{\frac{1}{w}(F^* \triangle F)}\,\middle|\, j\in[k],\, F \in \cF , \, F \ne F^* \right\}\right\}.
\end{gather*}
\end{thm}
\begin{proof}
By Lemma~\ref{lem:inv_min_cost_infinity_delta}, for each $j\in[k]$ there exists $\delta^j\geq 0$ such that $\p{\delta^j}{}{w}$ is an optimal deviation vector for the problem $( S, \cF, F^*, c^j,-\infty,+\infty, \| \cdot \|_{\infty, w})$. Our goal is to show that $\p{\delta}{}{w}$ is an optimal deviation vector for the multiple-cost variant, where $\delta\coloneqq \max \big\{ \delta^j \bigm| j\in[k] \big\}$.

Let $p$ be an optimal deviation vector. For ease of discussion, let us define
\begin{equation*}
d\coloneqq \max \left\{\frac{c^j(F^*) - c^j(F)}{\frac{1}{w}(F^* \triangle F)}\,\middle|\, j\in[k],\, F \in \cF , \, F \ne F^* \right\}.
\end{equation*}
The intuition behind the definition of this value is as follows: if the $c^j$-cost of a set $F$ is smaller than that of $F^*$, then the weighted $\ei$-norm of a feasible deviation vector is clearly lower bounded by the fraction appearing in the expression. 

If $F^*$ is a minimum $c^j$-cost member of $\cF $ for each $j\in[k]$, then we are clearly done. Otherwise, $\delta, d > 0$ holds, and it suffices to show $\delta = d$. Let $j\in [k]$ and $F \in \cF $, $F \ne F^*$ be arbitrary. Since $\delta \ge \delta^j$, Lemma~\ref{lem:inv_min_cost_infinity_observation_on_delta} implies that $\p{\delta}{}{w}$ is feasible, thus 
\begin{align*}
0 
{}&{}\ge 
(c^j - \p{\delta}{}{w})(F^*) - (c^j - \p{\delta}{}{w})(F) \\[2pt]
{}&{}= 
\left( c^j(F^*) - \sum_{s \in F^*} \frac{\delta}{w(s)} \right) - \left( c^j(F) - \sum_{s \in F \cap F^*} \frac{\delta}{w(s)} - \sum_{s \in F\setminus  F^*} \left( - \, \frac{\delta}{w(s)} \right) \right) \\[2pt]
{}&{}= 
c^j(F^*) - c^j(F) - \sum_{s \in F^* \triangle F} \frac{\delta}{w(s)} \\[2pt]
{}&{}=
c^j(F^*) - c^j(F) - \delta \cdot \tfrac{1}{w}(F^* \triangle F).
\end{align*}
This implies 
\begin{equation*}
\delta \ge \frac{c^j(F^*) - c^j(F)}{\frac{1}{w}(F^* \triangle F)},
\end{equation*}
hence $\delta \ge d$. To prove $\delta \le d$, it is enough to show that $\p{d}{}{w}$ is a feasible deviation vector.
 For any $j\in[k]$ and for any $F \in \cF$, $F \ne F^*$, we have
\begin{align*}
(c^j - \p{d}{}{w})(F^*) - (c^j - \p{d}{}{w})(F) 
{}&{}= 
c^j(F^*) - c^j(F) - d \cdot \tfrac{1}{w}(F^* \triangle F) \\
{}&{}\le 
c^j(F^*) - c^j(F) - \frac{c^j(F^*) - c^j(F)}{\frac{1}{w}(F^* \triangle F)} \cdot \tfrac{1}{w}(F^* \triangle F)\\
{}&{}= 
0,
\end{align*}
which means that $\p{d}{}{w}$ is indeed feasible.
\end{proof}

\section{Conclusions}
\label{sec:conc}

In this paper, we considered general minimum-cost inverse optimization problems in the constrained setting, i.e.\ with lower and upper bounds on the coordinates of the deviation vector. We provided simple, purely combinatorial algorithms for the weighted bottleneck Hamming distance and $\ei$-norm objectives. The algorithms follow a scheme that resembles Newton's algorithm, and find an optimal deviation vector in strongly polynomial when the bottleneck Hamming distance, and in pseudo-polynomial time when the $\ei$-norm is considered. For both objectives, we extended the results extend to inverse optimization problems with multiple cost functions.

Despite the extensive literature on inverse optimization, only few results are known when the desired deviation vector $p$ is required to be integral, see e.g.~\cite{ahmadian2018algorithms, frank2021inverse, frank2021discrete}. If $\ell$, $u$ and $c$ are integral vectors, then the deviation vector $\p{\delta}{\ell, u}{w}$ defined in the bottleneck Hamming distance case is integral independently from the choice of $\delta$, therefore the fractional and integral optimums coincide. For the unweighted $\ei$-norm objective, i.e.\ when in addition $w\equiv 1$ holds, the deviation vector $\p{\delta}{\ell, u}{w}$ might not be integral as we assign value $\delta$ to some of its coordinates. However, Lemmas~\ref{lem:inv_min_cost_infinity_delta} and~\ref{lem:inv_min_cost_infinity_observation_on_delta} together imply that if $\p{\delta}{\ell, u}{w}$ is an optimal fractional deviation vector, then $\p{\lceil\delta\rceil}{\ell, u}{w}$ is an optimal integral deviation vector. 

Though the proposed algorithm is capable of solving minimum cost inverse optimization problems in a very general setting, it naturally has its limitations. In particular, it is not suitable for handling deeper connections between the coordinates of the cost function. For example, if the underlying optimization problem is a minimum cost $s$-$t$ path problem in a directed graph with conservative arc-costs $c$, then it is not clear how to implement the algorithm as to maintain conservativeness throughout.

In an accompanying paper~\cite{berczi2023span}, we consider minimum-cost inverse optimization problems under the \emph{weighted span objective}, and provide a min-max characterization for the weighted span of an optimal deviation vector in the unconstrained setting, as well as an efficient algorithm for the constrained setting.

\paragraph{Acknowledgement.} The work was supported by the Lend\"ulet Programme of the Hungarian Academy of Sciences -- grant number LP2021-1/2021 and by the Hungarian National Research, Development and Innovation Office -- NKFIH, grant number FK128673.

\bibliographystyle{abbrv}
\bibliography{part1.bib}

\end{document}